\documentclass[reqno, 10pt, a4paper]{amsart}
\usepackage{amssymb,latexsym}
\usepackage{amsmath,color}
\usepackage{amsthm}
\usepackage{epsfig}
\usepackage{graphicx}
\usepackage{hyperref}
\usepackage{titletoc}
\usepackage{amsfonts}

\usepackage[cal=boondox]{mathalfa}

\numberwithin{equation}{section}
\newtheorem{theorem}{Theorem}[section]
\newtheorem{proposition}[theorem]{Proposition}
\newtheorem{lemma}[theorem]{Lemma}
\newtheorem{corollary}[theorem]{Corollary}

\theoremstyle{definition}

\newtheorem{remark}{Remark}[section]

\def\XXint#1#2#3{{\setbox0=\hbox{$#1{#2#3}{\int}$}
		\vcenter{\hbox{$#2#3$}}\kern-.5\wd0}}

\def\B{\mathbb{R}^n}
\def\R{\mathbb{R}_+^{n+1}}

\def\ou{\overline{u}}

\def\l{\lambda}
\def\s{(-\Delta)^s}
\def\ss{(-\Delta)^{\frac{s}{2}}}
\def\S{\mathbb{S}^{n-1}}
\def\e{\varepsilon}
\def\D{\Delta}
\def\vp{\varphi}
\def\RR{\mathbb{R}}
\def\ve{\varepsilon}

\begin{document}
	\title[Fractional Gelfand-Liouville]{Classification of stable solutions to a non-local Gelfand-Liouville equation}

	\author[A. Hyder]{Ali Hyder}
	\address{\noindent Ali Hyder, Department of Mathematics, Johns Hopkins University, Krieger Hall, Baltimore, MD, 21218}
	\email{ahyder4@jhu.edu}
	
	\author[W. Yang]{ Wen Yang}
	\address{\noindent Wen ~Yang,~Wuhan Institute of Physics and Mathematics, Chinese Academy of Sciences, P.O. Box 71010, Wuhan 430071, P. R. China}
	\email{wyang@wipm.ac.cn}
	
	
	\thanks{The first author is supported by the SNSF   Grant No. P400P2-183866.}
	\thanks{The second author is partially supported by NSFC No.11801550 and NSFC No.11871470.}

	\begin{abstract}
		We  study  finite Morse index  solutions to  the non-local  Gelfand-Liouville  problem
		\begin{equation*}
		(-\Delta)^su=e^u\quad\mathrm{in}\quad \B,
		\end{equation*} for every $s\in(0,1)$ and $n>2s$. Precisely, we  prove  non-existence of finite Morse index solutions whenever   the singular solution $$u_{n,s}(x)=-2s\log|x|+\log \left(2^{2s}\frac{\Gamma(\frac{n}{2})\Gamma(1+s)}{\Gamma(\frac{n-2s}{2})}\right)$$is unstable.  
	\end{abstract}
	\maketitle

	\section{Introduction}
	This paper is devoted to the study of   the following non-local Gelfand-Liouville equation
	\begin{equation}
	\label{fg-1}
	(-\Delta)^su=e^u\quad\mathrm{in}\quad \B.
	\end{equation}
	For  $s\in(0,1)$ the non-local operator $(-\D)^s $ is defined by  
	\begin{equation}
	\label{1.deff}
	\s u=c_{n,s}~\mathrm{P.V.}\int_{\B}\frac{u(x)-u(y)}{|x-y|^{n+2s}}dy,
	\end{equation}
	with $c_{n,s}$ being the normalizing constant 
	$$c_{n,s}=\frac{2^{2s}}{\pi^{n/2}}\frac{\Gamma(\frac{n+2s}{2})}{|\Gamma(-s)|}.
	$$
	To give a meaning of the equation \eqref{fg-1} we shall assume that $u\in L_{s}(\B)$ and  $e^u\in L^1_{\mathrm{loc}}(\B)$, where $L_{\mu}(\B)$ (for $\mu\geq-\frac{n}{2}$) is defined by
	$$L_\mu(\B):=\left\{u\in L^1_{\mathrm{loc}}(\B):\int_{\B}\frac{|u(x)|}{1+|x|^{n+2\mu}}dx<\infty  \right\}.$$
	Then \eqref{fg-1} is to be understood in the following sense:
	\begin{align}
	\label{weak-fg}
	\int_{\B}u(-\D)^s\vp dx=\int_{\B} e^u\vp dx\quad \mbox{for every }~\vp\in C_c^\infty(\B).
	\end{align}
	\smallskip

We recall that a solution $u$ to \eqref{fg-1} is said to be  stable in an open set  $\Omega\subseteq\B$ if
	\begin{equation}
	\label{1.stablecondition}
	\frac{c_{n,s}}{2}\int_{\B}\int_{\B}\frac{(\vp(x)-\vp(y))^2}{|x-y|^{n+2s}}dxdy  \geq\int_{\B}e^u\vp^2dx \quad \text{for every }\vp\in C_c^\infty(\Omega).
	\end{equation}
	While a solution is said to be a  finite Morse index solution of \eqref{fg-1}   if  it is stable  outside a compact set in $\B$. 
	
	In the particular case  $s=1$ and $n=2$,  equation \eqref{fg-1} is the well-known Liouville equation \cite{Liouville}, whose solutions can be represented in terms of locally injective meromorphic functions. Under the finite volume condition, that is $\int_{\RR^2}e^udx<\infty$, Chen-Li in their celebrated paper \cite{cl} classified all  solutions to \eqref{fg-1}  showing that, up to a translation,  they are radially symmetric   in $\RR^2$ (for the case $2s=n=1$ see \cite{DMR}).   It is known that these solutions are    finite Morse index solutions of \eqref{fg-1} in $\B$. Later on, Farina in \cite{f2} and Dancer-Farina in \cite{df} established  non-existence of stable solutions to \eqref{fg-1} for $2\leq n\leq 9$ and non-existence of finite Morse index solutions to \eqref{fg-1} for $3\leq n\leq 9.$ For the related cosmic string equation, Lane-Emden equations and systems, we refer the readers to \cite{ay,c1,c2,c3,dggw,ddw,ddww,f,fw,fw1,hx,wy,wy2,w1,w2} and references therein for the classification results of stable solutions and finite Morse index solutions.
	
In a recent work  Duong-Nguyen  \cite{dn} proved that equation \eqref{fg-1} has no regular stable solution for  $n<10s$. Their approach  is based on the Moser's iteration,  following the same spirit of \cite{f,f2,wy}.  However, these arguments does not work for $s\in (0,1)$ if either $n>10s$, or $u$ is stable outside a compact set. 

It  is known  (see e.g. \cite[Proposition 3.2]{r1}) that the function
\begin{equation}
\label{1.singular}
u_{n,s}(x):=-2s\log|x|+\log \lambda_{n,s},\quad \lambda_{n,s}:=2^{2s}\dfrac{\Gamma(\frac n2)\Gamma(1+s)}{\Gamma(\frac{n-2s}{2})}
\end{equation}
is a singular solution to \eqref{fg-1}. It is interesting to note that the function $e^{u_{n,s}}$ is precisely the Hardy weight for the operator $(-\D)^s$. More precisely, the following   Hardy inequality holds (see \cite[Theorem 2.9]{y} and \cite{h}):  
	\begin{equation*}
	\frac{c_{n,s}}{2}\int_{\B}\int_{\B}\dfrac{(\psi(x)-\psi(y))^2}{|x-y|^{n+2s}}\geq\Lambda_{n,s}\int_{\B}|x|^{-2s}\psi^2dx \quad \text{for every }\psi\in C_c^{\infty}(\B),
	\end{equation*}
	where the optimal constant $\Lambda_{n,s}$ is given by
	\begin{equation}
	\label{1.hconst}
	\Lambda_{n,s}=2^{2s}\dfrac{\Gamma^2(\frac{n+2s}{4})}{\Gamma^2(\frac{n-2s}{4})}.
	\end{equation}
This shows that the singular solution  $u_{n,s}$ is stable if and only if
	\begin{equation}
	\label{1.stable}
	\dfrac{\Gamma(\frac n2)\Gamma(1+s)}{\Gamma(\frac{n-2s}{2})}\leq \dfrac{\Gamma^2(\frac{n+2s}{4})}{\Gamma^2(\frac{n-2s}{4})}.
	\end{equation}
	As a consequence of \eqref{1.stable}, we get that (see \cite[Proposition 3.2]{r1} and \cite[Theorem 1.1]{r2})
	\begin{enumerate}
		\item [(1).] If $n\leq 7$, then $u_{n,s}$ is unstable for all $s\in(0,1).$
		\item [(2).] If $n=8$, then $u_{n,s}$ is stable if and only if $s\leq0.28206....$
		\item [(3).] If $n=9$, then $u_{n,s}$ is stable if and only if $s\leq 0.63237....$
		\item [(4).]If $n\geq10$, then $u_{n,s}$ is stable for all $s\in(0,1).$
	\end{enumerate}
    The stability condition \eqref{1.stable} for the solution $u_{n,s}$ suggests that equation \eqref{fg-1} might not admit any stable solution if the following inequality holds: 
	\begin{equation}
	\label{1.stable1}
	\dfrac{\Gamma(\frac n2)\Gamma(1+s)}{\Gamma(\frac{n-2s}{2})}> \dfrac{\Gamma^2(\frac{n+2s}{4})}{\Gamma^2(\frac{n-2s}{4})}.
	\end{equation}
	It is worth pointing out that the condition   $n<10s$ in \cite{dn}  implies \eqref{1.stable1}.  Interestingly, in this range, stable solutions are smooth.  
	
	\begin{theorem}
		\label{th1.2}
		Let $\Omega$ be an open set in $\B$ and $s\in (0,1)$. If $u \in L_s(\B)\cap\dot  H_{\mathrm{loc}}^s(\Omega)$ is stable in $\Omega$ and $n<10s$, then $e^u\in L^p_{\mathrm{loc}}(\Omega)$ for every $p\in[ 1,5)$. In particular,  $u$ is smooth in $\Omega$.  
	\end{theorem}

	Here the function space $\dot H^s(\Omega)$ is defined by $$\dot H^s(\Omega):=\left\{ u\in L^2_{\mathrm{loc}}(\Omega): \int_{\Omega}\int_{\Omega}\frac{|u(x)-u(y)|^2}{|x-y|^{n+2s}}dxdy<\infty\right\}.$$ The notation $\dot H^s_{\mathrm{loc}}(\B)$ will be used to denote the set of  all functions which are in $\dot H^s(\Omega)$ for every bounded open set $\Omega\subset\B$.  
	
	\smallskip
	
	The proof of Theorem \ref{th1.2} is based on the Farina type estimates for stable solutions. To this end we use the Caffarelli-Silvestre  \cite{cs}  extension  $\ou$ of $u$ on the upper-half space $\R$: 
	\begin{equation}
	\label{Poisson-repre}
	\overline{u}(X)=\int_{\B}P(X,y)u(y)dy,\quad X=(x,t)\in \B\times (0,\infty),
	\end{equation}
	where
	\begin{equation*}
	P(X,y)=d_{n,s}\dfrac{t^{2s}}{|(x-y,t)|^{n+2s}},
	\end{equation*}
	and $d_{n,s}>0$ is a normalizing constant so that   $\int_{\B}P(X,y)dy=1$. Notice that $\ou$ is well-defined as $u\in L_s(\B)$.  Moreover, $t^\frac{1-2s}{2}\nabla \ou\in L^{2}_{\mathrm{loc}}(\Omega\times[0,\infty))$ whenever $u\in \dot H^s(\Omega)$.  The equation  \eqref{weak-fg} in terms of $\ou$ now reads 
	\begin{equation}
	\label{1.weak}
	\int_{\R}t^{1-2s}\nabla\ou\cdot\nabla\Phi(x,t)dxdt=\kappa_s\int_{\B}e^u\vp  dx\quad\mbox{for every}~\Phi\in C_c^\infty(\R),
	\end{equation}
	where $\vp (x)=\Phi(x,0)$ and $\kappa_s=\frac{\Gamma(1-s)}{2^{2s-1}\Gamma(s)}.$
	
\smallskip
	
	
	Concerning the non-existence of finite Morse index solution we prove: 
	
	\begin{theorem} \label{th1.1} 	
	Assume that $n>2s$ and $s\in (0,1)$. If  \eqref{1.stable1} holds then \eqref{fg-1} does not admit a finite Morse index solution $u\in L_s(\B)\cap W^{1,2}_{\mathrm{loc}}(\B)$ satisfying  $e^u\in L_{\mathrm{loc}}^2(\B)$.
	\end{theorem}

The hypothesis $e^u\in L^2_{\mathrm{loc}}(\B)$ and the regularity assumption  $u\in W^{1,2}_{\mathrm{loc}}(\B)$ in  the above theorem can be weakened by simply assuming $u\in W^{1,q}_{\mathrm{loc}}(\B) $ with $q$ slightly bigger than $\frac54$, see Remark \ref{re1.1}.  These assumptions will be used only to derive the following monotonicity formula, which is a crucial tool in proving Theorem \ref{th1.1}.  

   \begin{theorem}
	\label{th4.2}
Let $u\in L_{s}(\B)\cap W^{1,2}_{\mathrm{loc}}(\B)$ be a solution to  \eqref{fg-1}. Assume that    $e^u\in L^2_{\mathrm{loc}}(\B)$. For $x_0\in\partial\R$ and $\l>0$, we define
	\begin{equation}
	\label{1.monotonicity}
		\begin{aligned}
		E(\ou,x_0,\l):=~&\lambda^{2s-n}\left(\frac12\int_{B^{n+1}(x_0,\l)\cap \R}t^{1-2s}|\nabla\ou|^2dxdt-\kappa_s\int_{B(x_0,\l)}e^{\ou}dx\right)\\
		&+2s\l^{2s-n-1}\int_{\partial B^{n+1}(x_0,\l)\cap\R}t^{1-2s}(\ou+2s\log r)d\sigma.
		\end{aligned}
		\end{equation}
		Then $E$ is a nondecreasing function of $\l$. Furthermore,
		\begin{equation*}
		\frac{dE}{d\l}=\l^{2s-n}\int_{\partial B^{n+1}(x_0,\l)\cap\R}t^{1-2s}\left(\frac{\partial\ou}{\partial r}+\frac{2s}{r}\right)^2d\sigma,
		\end{equation*}
		where $B^{n+1}(x_0,\l)$ denotes the Euclidean ball in $\mathbb{R}^{n+1}$ centered at $x_0$ of radius $\l$, $\sigma$ is the $n$-dimensional surface measure on $\partial B^{n+1}(x_0,\l),$ $X=(x,t)\in\R$, $r=|(x-x_0,t)|$ and $\partial_r=\nabla\cdot\frac{X-(x_0,0)}{r}$ is the corresponding radial derivative.
	\end{theorem}

	\medskip
	
	

	The study of the Lane emden equation by using the  monotonicity formula method goes back to a series of works \cite{ddw,ddww,w0}. Very recently, Wang \cite{w3} applied  such method to study the stable solutions of Toda systems. Compared with the polynomial nonlinearity, the control on the integral of the weighted square term and the boundary integral of the linear term (the first and third term in the motonicity formula \eqref{1.monotonicity}) turns to be more difficult for the exponential nonliearity. Due to this difficulty, Wang  \cite{w3} used the $\epsilon$-regularity theory to exclude the case that the boundary term going to infinity in the procedure of performing the blowing down analysis. In this paper, instead of using the $\epsilon$-regularity theory, we shall use a more straightforward way to estimate  each term in \eqref{1.monotonicity}.  {We believe that this part of the analysis can be adapted for studying some other problems with exponential nonlinearity.}

	Before ending the introduction let us briefly mention our strategy  for the proof of Theorem \ref{th1.1}.  We consider a family of rescaled solutions $$u^\l(x)=u(\lambda x)+2s\log\lambda,\quad \lambda\geq1.$$  Together with the above  monotonicity formula, Farina type estimates (Proposition \ref{prop-2.6}) and the integral representation formula (Lemma \ref{u-representation}) we  prove the convergence $u^\l\to u^\infty$ as $\lambda\to\infty$. Then the limit function $u^\infty$ is a stable solution of \eqref{fg-1}. Again,     the monotonicity formula is crucial to show that  $u^\infty$   is   homogeneous in $\B$.   The proof then follows from the non-existence result in Theorem \ref{th3.1}.

%

\medskip

	The current paper is organized as follows. First we obtain some decay estimates and integral presentation for the solutions in Section 2. In Section 3, we derive the higher order integrability of $e^u$ and prove Theorem \ref{th1.2}. {While in section 4, we first study the nonexistence of homogeneous stable soluion under \eqref{1.stable1} and the monotonicity formula, then} we consider the family of solutions arising from the blow-down analysis and use it to reduce the stable solutions to stable homogeneous solution, from which Theorem \ref{th1.1} is established.

	\bigskip
	\begin{center}
		Notations:
	\end{center}
	\begin{enumerate}
		\item [$B_R^{n+1}$] \quad the ball centered at $0$ with radius $R$ in dimension $(n+1)$.
		\smallskip
		\item [$B_R$] \quad the ball centered at $0$ with radius $R$ in dimension $n$.
		\smallskip
		\item [$B^{n+1}(x_0,R)$] \quad  the ball centered at $x_0$ with radius $R$ in dimension $(n+1)$.
		\smallskip
		\item [$B(x_0,R)$] \quad the ball centered at $x_0$ with radius $R$ in dimension $n$.
		\smallskip
		\item [$X=(x,t)$] \quad represent  points in $\mathbb{R}_+^{n+1}=\B\times[0,\infty).$
		\smallskip
		\item [$\ou$] \quad $s$-harmonic extension  of $u$ on $\mathbb{R}_+^{n+1}$.
		\item [$C$] \quad  a generic positive constant which may change from line to line.
		\item [$C(r)$] \quad a positive constant depending on $r$ and may change from line to line.
		\item [$\sigma$] \quad the $n$-dimensional Hausdorff measure restriced to $\partial B^{n+1}(x_0,r)$.
	\end{enumerate}

	\medskip
	\section{Preliminary estimates}
	In this section  we use the stability condition outside a compact set to derive  energy estimates on $e^u$  and the integral representation formula for $u$.
	
	\begin{lemma}
		\label{leh.1}
		Let $u$ be a solution to \eqref{fg-1} for some $n>2s$. Suppose that $u$ is stable outside a compact set. Then
		\begin{equation}
		\label{h.7}
		\int_{B_r}e^udx\leq Cr^{n-2s}\quad\text{for every } r\geq1.
		\end{equation}
	\end{lemma}	
	
	\begin{proof} Let $R\gg1$ be such that $u$ is stable on $\B\setminus B_R$. We fix two  smooth cut-off functions $\eta_R$ and $\varphi$ in $\B$ such that
		\begin{align*}
		\eta_R(x)=\left\{\begin{array}{ll}
		0\quad&\text{for }|x|\leq R\\
		\\
		1\quad&\text{for }|x|\geq 2R
		\end{array}\right. ,\quad
		\varphi(x)=\left\{\begin{array}{ll}
		1\quad&\text{for }|x|\leq 1\\
		\\
		0\quad&\text{for }|x|\geq 2  \end{array}\right..
		\end{align*}
		Setting $\psi(x)=\eta_R(x)\varphi(\frac xr)$ with $r\geq 1$ we see that $\psi$ is a good test function for the stability condition \eqref{1.stablecondition}. Hence,
		\begin{equation*}
		\int_{B_r}e^udx\leq C+\int_{\B}|\ss\psi|^2dx\leq C+Cr^{n-2s}\leq Cr^{n-2s},
		\end{equation*}
		where we used that $n>2s$ and $r\geq1.$
	\end{proof}

	It is not difficult to see that if $u$ is a solution to \eqref{fg-1}, then
	$$u^\l(x)=u(\l x)+2s\log\l$$
	establishes a family of solutions to \eqref{fg-1}. In addition,  $u$ is stable on $\B\setminus B_R$ if and only if $u^\l$ is  stable  on $\B\setminus B_{\frac R\l}$.
	
	 As a simple consequence of Lemma \ref{leh.1} we obtain the following corollary:
	\begin{corollary}
		\label{cor2.2}
		Suppose $n>2s$ and $u$ is a solution of \eqref{fg-1} which is stable outside a compact set.  Then there exists $C>0$ such that
		\begin{align}
		\label{f.1}
		\int_{B_r}e^{u^\l}dx\leq Cr^{n-2s}\quad\text{for every }\lambda\geq 1,\,r\geq1.
		\end{align}
	\end{corollary}
	
	 With the help of above decay estimate on $e^{u^\l}$, we  show that $e^{u^\l}\in L_{\mu}(\B)$ for some $\mu<0$:
	
	\begin{lemma}
		\label{leh.3}
		For $\delta>0$  there exists $C=C(\delta)>0$ such that
		\begin{equation*}
		\int_{\B}\frac{e^{u^\l}}{1+|x|^{n-2s+\delta}}dx \leq C\quad\text{for every }\lambda\geq1.
		\end{equation*}
	\end{lemma}
	\begin{proof}
		It suffices to show that
		\begin{equation*}
		\int_{\B\setminus B_1}\frac{e^{u^\l}}{1+|x|^{n-2s+\delta}}dx\leq C.
		\end{equation*}
		By Corollary \ref{cor2.2}
		\begin{equation*}
		\int_{\B\setminus B_1}\frac{e^{u^\l}}{1+|x|^{n-2s+\delta}}dx
		=\sum_{i=0}^\infty\int_{2^i\leq |x|<2^{i+1}}\frac{e^{u^\l}}{1+|x|^{n-2s+\delta}}dx
		\leq C\sum_{i=0}^\infty\frac{1}{2^{i\delta }}<\infty.
		\end{equation*}	
		Hence we finish the proof.
	\end{proof}

	We set
	\begin{equation}
	\label{f.3}
	v^\l(x):=c(n,s)\int_{\B}\left(\frac{1}{|x-y|^{n-2s}}-\frac{1}{(1+|y|)^{n-2s}}\right)e^{u^\l(y)}dy,
	\end{equation}
	where $c(n,s)$ is chosen such that
	$$c(n,s)(-\Delta)^s\frac{1}{|x-y|^{n-2s}}=\delta(x-y).$$
	It is not difficult to see that $v^\l\in L^1_{\mathrm{loc}}(\B)$. In addition, we have that $v^\lambda\in L_s(\B)$. This is the conclusion of the following lemma: 
	
	\begin{lemma}
		\label{lef.1}
We have 
		\begin{equation}
		\label{f.4}
		\int_{\B}\frac{|v^\l(x)|}{1+|x|^{n+2s}}dx\leq C\quad \mbox{for every }\quad \l\geq1.
		\end{equation}
		Moreover, for every $R>0$ we have
		\begin{equation}
		\label{f.5}
		v^\l(x)=c(n,s)\int_{B_{2R}}\frac{1}{|x-y|^{n-2s}}e^{u^\l(y)}dy+O_R(1)\quad \mbox{for}~\l\geq 1,~|x|\leq R.
		\end{equation}
	\end{lemma}
	
	
	\begin{proof}
		Let us first set
		$$f(x,y)=\left(\frac{1}{|x-y|^{n-2s}}-\frac{1}{(1+|y|)^{n-2s}}\right),$$
		and estimate the term
		$$E(y):=\int_{\B}\frac{|f(x,y)|}{1+|x|^{n+2s}}dx\quad\mbox{for}\quad |y|\geq 2.$$
		We split $\B$ into
		$$\B=\bigcup_{i=1}^4A_i,$$
		where
		\begin{equation*}
		A_1:=B_{|y|/2},~ A_2:=\B\setminus B_{2|y|},~
		A_3:=B(y,|y|/2),~A_4:=\B\setminus (A_1\cup A_2\cup A_3).
		\end{equation*}
Then the following estimates hold
		\begin{equation*}
		|f(x,y)|\leq C
		\begin{cases}
		\dfrac{1+|x|}{|y|^{n-2s+1}} \quad &\mbox{if}~x\in A_1,\\
		\dfrac{1}{|y|^{n-2s}}  \quad &\mbox{if}~x\in A_2\cup A_4,\\
		\dfrac{1}{|y|^{n-2s}}+\dfrac{1}{|x-y|^{n-2s}} \quad &\mbox{if}~x\in A_3.
		\end{cases}
		\end{equation*}
We write 
		\begin{equation*}
		E(y)=\sum_{i=1}^4 E_i(y),\quad E_i(y)=\int_{A_i}\dfrac{|f(x,y)|}{1+|x|^{n+2s}}dx.
		\end{equation*}
		For $E_1(y)$, one has
		\begin{equation*}
		|E_1(y)|\leq\dfrac{C}{|y|^{n-2s+1}}\int_{A_1}\dfrac{1+|x|}{1+|x|^{n+2s}}dx\leq \dfrac{C}{|y|^{n-2s+\gamma}}
		\end{equation*}
		where
		\begin{equation*}
		\gamma=\begin{cases}
		1  ~&\mbox{if}~2s>1,\\
		2s ~&\mbox{if}~2s\leq 1.
		\end{cases}
		\end{equation*}
		Next, we bound the second  and the fourth term
		\begin{equation*}
		|E_2(y)|+|E_4(y)|\leq\frac{C}{|y|^{n-2s}}\int_{\B\setminus A_1}\frac{1}{1+|x|^{n+2s}}dx\leq \frac{C}{|y|^n}.
		\end{equation*}
		While for the last term $E_3(y)$, we notice that $|x|\sim |y|$ for $x\in A_3$.  Therefore, 
		\begin{equation*}
		\begin{aligned}
		|E_3(y)|&\leq \frac{C}{|y|^n}+\frac{C}{|y|^{n+2s}}
		\int_{A_3}\frac{1}{|x-y|^{n-2s}}dx\\
		&\leq\frac{C}{|y|^n}+\frac{C}{|y|^{n+2s}}
		\int_{B_{3|y|}}\frac{1}{|x|^{n-2s}}dx\leq\frac{C}{|y|^n}.
		\end{aligned}
		\end{equation*}
		Thus, there exists $\gamma>0$ such that
		\begin{equation*}
		|E(y)|\leq\dfrac{C}{|y|^{n-2s+\gamma}}\quad \mbox{for}\quad |y|\geq 2.
		\end{equation*}
		Therefore, by Lemma \ref{leh.3}   we get that
		\begin{equation}
		\label{f.111}
		\begin{aligned}
		\int_{\B}\dfrac{|v^\l(x)|}{1+|x|^{n+2s}}dx\leq~&
		C\int_{\mathbb{R}^n\setminus B_2}e^{u^\l(y)}|E(y)|dy
		+C\int_{\B}\int_{B_2}\frac{e^{u^\l(y)}}{|x-y|^{n-2s}(1+|x|^{n+2s})}dydx\\
		&+C\int_{\B}\int_{B_2}\frac{e^{u^\l(y)}}{((1+|y|)^{n-2s})(1+|x|^{n+2s})}dydx
		<+\infty.
		\end{aligned}
		\end{equation}
		This finishes the proof of \eqref{f.4}. 
		
		To prove \eqref{f.5} we notice that 
		\begin{equation*}
		\left|\frac{1}{|x-y|^{n-2s}}-\frac{1}{(1+|y|)^{n-2s}}\right|\leq C\frac{1+|x|}{|y|^{n-2s+1}} \quad\text{for }|x|\leq R,\, |y|\geq 2R.
		\end{equation*}
		Then
		\begin{equation*}
		\begin{aligned}
		\int_{|y|\geq 2|x|}\left|\frac{1}{|x-y|^{n-2s}}-\frac{1}{(1+|y|)^{n-2s}}\right|e^{u^\l(y)}dy
		\leq ~&C(1+|x|)\int_{|y|\geq2|x|}\frac{e^{u^\l(y)}}{|y|^{n-2s+1}}dy\\
		\leq ~&C(1+|x|),
		\end{aligned}
		\end{equation*}
		where we used Lemma \ref{leh.3} in the last inequality. Then \eqref{f.5} follows immediately.
	\end{proof}
	
	\begin{lemma}\label{u-representation}For some $c_\lambda\in\RR$  we have
		\begin{equation}
		\label{f.11}
		u^\l(x)=c(n,s)\int_{\B}\left(\frac{1}{|x-y|^{n-2s}}-\frac{1}{(1+| y|)^{n-2s}}\right)e^{u^\l(y)}dy+c_\l.
		\end{equation}
	\end{lemma}\begin{proof} According to the definition of $v^\l$, we can easily see that $h^\lambda:=u^\l-v^\l$ is a $s$-harmonic function in $\B$. In the spirit of   \cite[Lemma 2.4]{ha} one can show that   $h^\l $ is either a constant, or   a polynomial of degree one. In order to    rule out the second possibility,   we first show  
		\begin{equation}
		\label{f.10}
		v^\l(x)\geq -C\log|x|\quad\mbox{for}~|x|~\mbox{large},
		\end{equation} for some constant $C>0$. 
		Indeed,  by \eqref{f.1} we have
		\begin{align*}
		v^\l(x)&\geq -C\left(1+\int_{2\leq|y|\leq2|x|}\frac{1}{(1+|y|)^{n-2s}}
		e^{u^\l(y)}dy\right)\\
		&\geq -C-C\sum_{i=1}^{[\log(2|x|)]}\int_{2^i\leq |y|\leq 2^{i+1}}\frac{1}{(1+|y|)^{n-2s}}e^{u^\l(y)}dy\\
		&\geq -C- C\log|x|.
		\end{align*}
		where $[x]$ denotes the integer part of $x$. Hence, \eqref{f.10} is proved. Therefore
		$$u^\l(x)\geq -C\log|x| +h^\l(x).$$
		Using \eqref{f.1} again, we derive that $h^\l  \equiv c_\l$ for some constant $c_\l\in\mathbb{R}$.
	\end{proof}

	\medskip
	\section{Higher order integrability}
	In this section we prove a higher order integrability of the nonlinearity $e^u$ on the region where $u$ is stable.  More precisely, we establish the following Farina's estimate, see \cite{f,df} for the classical case.
	\begin{proposition}  
		\label{prop-2.6}
		Let $u\in L_s(\B)\cap \dot H^s_{\mathrm{loc}}(\B)$ be a solution to \eqref{fg-1}. Assume that $u$ is stable on $\B\setminus B_R$ for some $R>0$. Then for every $p\in [1,\min\{5,1+\frac{n}{2s}\})$ there exists $C=C(p)>0$ such that for $r$ large
		\begin{align}
		\label{est-p-1}
		\int_{B_{2r}\setminus B_{r}}e^{pu}dx\leq Cr^{n-2ps}.
		\end{align}
In particular, 
		\begin{itemize}
			\item[(i)] for $|x|$ large,
			\begin{align}
			\label{est-p-2} \int_{B_{|x|/2}(x)}e^{pu (y)}dy\leq C(p)|x|^{n-2ps} \quad\text{for every }p\in[1,\min\{5,1+\frac{n}{2s}\}),
			\end{align}
			\item[(ii)] for $r$ large
			\begin{align} \label{est-p-3}\int_{B_r\setminus  B_{2R}}e^{pu}dx\leq C(p)r^{n-2ps}\quad \text{for  every }p\in[1,\min\{5, \frac{n}{2s}\}).
			\end{align}
		\end{itemize}
	\end{proposition}
	\medskip
	
	Before proving Proposition \ref{prop-2.6}, we shall first apply it to derive an upper bound  on $u$ in the following lemma:
	\begin{lemma}
		\label{lemma-3.2}
		Suppose that $n<10s$. If $u$ satisfies   the assumptions in Proposition \ref{prop-2.6},   then for $|x|$ large we have
		$$u(x)\leq -2s\log|x|+C.$$
	\end{lemma}
	\begin{proof}
		Set
		$$w (r):=c(n,s)\int_{|y|\leq r}\frac{1}{(1+|y|)^{n-2s}}e^{u(y)}dy,$$
		where $c(n,s)$ is as in \eqref{f.3}. It is easy to see that $w(r)$ is locally bounded by Lemma \ref{leh.1}. Next,
		we claim that for $|x|$ large we have \begin{align}
		\label{u-u}
		u(x)=-w(|x|)+O(1).
		\end{align}
		Using the estimate
		$$\left| \frac{1}{|x-y|^{n-2s}}-\frac{1}{(1+|y|)^{n-2s}}\right|\leq C\frac{|x|}{|y|^{n-2s+1}},\quad |x|\geq 1, \,|y|\geq \frac32|x|, $$
		we get
		\begin{align*}
		&\int_{|y|\geq \frac32|x|} \left| \frac{1}{|x-y|^{n-2s}}-\frac{1}{(1+|y|)^{n-2s}}\right|e^{u(y)}dy\\&\leq C|x|\int_{|y|\geq |x|} \frac{e^{u(y)}}{|y|^{n-2s+1}}dy \\&\leq C|x|\sum_{i=0}^\infty\frac{1}{(2^i|x|)^{n-2s+1}}\int_{2^i|x|\leq |y|\leq 2^{i+1}|x|}e^{u(y)}dy\\&\leq C,
		\end{align*}
		where \eqref{h.7} is used. Choosing $p\in (1,\min\{5,\frac{n}{2s}+1\})$ such that $(n-2s)p'<n$ (this is possible as $n<10s$), and together with \eqref{est-p-2}
		\begin{align*} \int_{B_\frac{|x|}{2}(x)}\frac{e^{u(y)}}{|x-y|^{n-2s}}dy &\leq \left(  \int_{B_\frac{|x|}{2}(x)}\frac{dy}{|x-y|^{(n-2s)p'}} \ \right)^\frac{1}{p'} \left(  \int_{B_\frac{|x|}{2}(x)}e^{pu(y)} dy\ \right)^\frac1p\\&\leq C.
		\end{align*}
		Using the above estimates, \eqref{h.7} and the representation formula \eqref{f.11} we get \eqref{u-u}. Then for $r$ large, using Jensens inequality we see that \begin{align*}
		\log\left( \frac{1}{|B_r|}\int_{B_r}e^{u(y)}dy\right)
		&\geq \frac{1}{|B_r|}\int_{B_r}u(y)dy\\ &=-\frac{1}{|B_r|}\int_{B_r}w(|y|)dy+O(1)\\
		&\geq -w(r)+O(1),
		\end{align*}
		where we used 	the fact that $w$ is monotone increasing in the last inequality. Thus, by \eqref{h.7}  we obtain
		$$-w(r)\leq \log(C r^{-2s})+O(1)=-2s\log r+O(1).$$
		This proves the lemma.
	\end{proof}
	
	The rest of this section is devoted to the proof of Proposition \ref{prop-2.6} and Theorem \ref{th1.2}. First, we notice that the stability condition \eqref{1.stablecondition} can be extended to  $\ou$.
	More precisely, if $u$ is stable in $\Omega$ then 
	\begin{equation}
	\label{h.2}
	\int_{\R}t^{1-2s}|\nabla\Phi|^2dxdt\geq \kappa_s\int_{\B}e^u\vp ^2dx, 
	\end{equation} 
	for every $\Phi\in C_c^\infty(\overline\R)$ satisfying $\vp(\cdot):=\Phi(\cdot,0)\in C_c^\infty(\Omega)$. 
	Indeed, if  $\overline{\vp }$ is the $s$-harmonic extension of $\vp $,   we have
	\begin{align*}
	\int_{\R}t^{1-2s}|\nabla\Phi|^2dxdt &\geq \int_{\R}t^{1-2s}|\nabla\overline\vp |^2dxdt\\&
	=\kappa_s\int_{\B}\vp (-\Delta)^s\vp  dx\geq \kappa_s\int_{\B}e^u\vp ^2dx.
	\end{align*}
	
	\medskip
	Before we study the equation \eqref{1.weak}, we present the following lemma which will be used later.
	
	\begin{lemma}
		\label{lem-2.7}Let $e^{\alpha u}\in L^1( \Omega) $ for some $\Omega\subset \B$. Then $t^{1-2s}e^{\alpha\ou}\in L^1_{\mathrm{loc}}(\Omega\times [0, \infty))$.
	\end{lemma}
	\begin{proof}
		Let $\Omega_0\Subset\Omega$ be fixed.   Since $ u\in  L_s(\B)$, we have for $x\in\Omega_0$ and $t\in(0,R)$
		$$\ou(x,t)\leq C+\int_{\Omega}u(y)P(X,y)dy=C+\int_{\Omega}g(x,t)u(y)\frac{P(X,y)dy}{g(x,t)},$$
		where
		$1\geq  g(x,t):=\int_{\Omega}P(X,y)dy\geq C$ for some positive constant $C$ depending on $R$, $\Omega_0$ and $\Omega$ only. Therefore, by Jensen's inequality
		\begin{align*}
		\int_{\Omega_0}e^{\alpha \ou(x,t)}dx
		&\leq C \int_{\Omega_0}\int_{\Omega}e^{\alpha g(x,t)u(y)}P(X,y)dydx\\
		&\leq C\int_{\Omega}\max\{e^{\alpha u(y)},1\}\int_{\Omega_0}P(X,y)dxdy
		\\&\leq C+C\int_{\Omega}e^{\alpha u(y)}dy,
		\end{align*}
		where the constant $C$ depends on $R,~\Omega_0$ and $\Omega$, but not on $t$. Hence,
		\begin{align*}
		\int_{\Omega_0\times(0,R)} t^{1-2s}e^{\alpha \ou(x,t)}dxdt \leq
		\int_0^R t^{1-2s}\int_{\Omega_0}e^{\alpha \ou(x,t)}dxdt <\infty.
		\end{align*}
		This finishes the proof.
	\end{proof}
	
	The following lemma is crucial for the proof of Proposition \ref{prop-2.6}.
	
	\begin{lemma}
		\label{lem-0.2}
		Let $u\in L_s(\B)\cap \dot H^s_{\mathrm{loc}}(\B)$ be a solution to \eqref{fg-1}. Assume that $u$ is stable in $\Omega\subseteq\B$. Let    $\Phi \in C_c^\infty({\overline\R})$ be of the form $\Phi (x,t) =\vp(x)\eta(t)$  for some $\vp\in C_c^\infty(\Omega)$  and $\eta \equiv 1$ on $[0,1]$.  Then for every $0<\alpha<2$  we have
		\begin{equation}
		\label{h.6}
		\begin{aligned}
		(2-\alpha)\kappa_s\int_{\B}e^{(1+2\alpha) u}\vp^2dx
		&\leq 2\int_{\R}t^{1-2s}e^{2\alpha\bar u}|\nabla \Phi|^2dxdt\\
		&\quad  -\frac12\int_{\R}e^{2\alpha\bar u}\nabla\cdot[t^{1-2s}\nabla\Phi^2]dxdt . \end{aligned}
		\end{equation}
	\end{lemma}

	\begin{proof}
		For $k\in\mathbb{N}$ we set $\ou_k:=\max\{\ou,k\}$, and let $u_k$ be the restriction of $\ou_k$ on $\B$.  It is easy to see that $e^{2\alpha\ou_k}\Phi^2$ is a good test function in \eqref{1.weak}. Therefore,
		\begin{equation}
		\begin{aligned}
		\label{7}
		&\kappa_s\int_{\B}e^ue^{2\alpha u_k}\vp^2dx\\
		&=2\alpha\int_{\R}t^{1-2s} e^{2\alpha\ou_k}\Phi^2\nabla\ou\cdot \nabla \ou_kdxdt +\int_{\R}t^{1-2s}e^{2\alpha\ou_k}\nabla\ou\cdot \nabla \Phi ^2dxdt\\
		&=2\alpha\int_{\R}t^{1-2s} e^{2\alpha\ou_k}\Phi^2\ | \nabla \ou_k|^2dxdt +\int_{\R}t^{1-2s}e^{2\alpha\ou_k}\nabla\ou\cdot \nabla \Phi^2dxdt.
		\end{aligned}
		\end{equation}
		Now we assume that $t^{1-2s}e^{2(\alpha+\ve)\ou}\in L^1_{\mathrm{loc}}(\Omega\times [0,\infty))$ for some $\ve>0$. Then by Lemma \ref{lem-decay} below, up to a subsequence,
		$$\kappa_s\int_{\B}e^ue^{2\alpha u_k}\vp^2dx=(2\alpha+o(1))\int_{\R}t^{1-2s} e^{2\alpha\ou_k}\Phi^2\ | \nabla \ou_k|^2dxdt+O(1).$$
		Taking $e^{\alpha\ou_k} \Phi$  as a test function in the stability inequality  \eqref{h.2}
		\begin{equation}
		\label{8}
		\begin{aligned}
		\kappa_s \int_{\B}e^ue^{2\alpha u_k}\vp^2dx &\leq \alpha^2 \int_{\R}t^{1-2s}\Phi^2e^{2\alpha\ou_k}|\nabla\ou_k|^2dxdt+ \int_{\R}t^{1-2s}e^{2\alpha\ou_k}|\nabla  \Phi|^2dxdt
		\\ &\quad +\frac12 \int_{\R}t^{1-2s} \nabla e^{2\alpha \ou_k}\nabla\Phi ^2dxdt
		\\ &=  \alpha^2 \int_{\R}t^{1-2s}\Phi^2e^{2\alpha \ou_k}|\nabla \ou_k|^2dxdt+ \int_{\R}t^{1-2s}e^{2\alpha\ou_k}|\nabla \Phi|^2dxdt
		\\ &\quad -\frac12 \int_{\R}e^{2\alpha\ou_k} \nabla\cdot [t^{1-2s}\nabla \Phi^2]dxdt,
		\end{aligned}
		\end{equation}
		where the last equality follows by integration by parts. Notice that the boundary term vanishes as $\eta(t)=1$ on $[0,1]$. From the above two relations we obtain
		\begin{equation}
		\label{9}
		\begin{aligned}
		&(2-\alpha-\ve)\kappa_s\int_{\B}e^ue^{2\alpha u_k}\vp^2dx
		\\&\leq(-2\alpha\ve+o(1))\int_{\R}t^{1-2s} e^{2\alpha\ou_k}\Phi^2|\nabla \ou_k|^2dxdt+O(1)
		\\&\quad+2\int_{\R}t^{1-2s}e^{2\alpha\ou_k}|\nabla \Phi|^2dx -\int_{\R}e^{2\alpha\ou_k}\nabla\cdot[t^{1-2s}\nabla \Phi^2]dxdt.
		\end{aligned}
		\end{equation}
		Again, as $\eta=1$ on $[0,1]$,  the second term   in the right hand side of the  following expression is identically zero for $0\leq t\leq 1$:
		$$\nabla\cdot[t^{1-2s}\nabla \Phi^2]=t^{1-2s}\eta^2\D_x\vp^2+\vp^2\partial_t(t^{1-2s}\partial_t\eta^2).$$
		Therefore, using Lemma \ref{lem-2.7}
		$$\left|\int_{\R}t^{1-2s}e^{2\alpha\ou_k}|\nabla \Phi|^2dxdt\right| +\left|\int_{\R}e^{2\alpha\ou_k}\nabla\cdot[t^{1-2s}\nabla \Phi^2]dxdt\right|\leq C.$$
		Thus,
		\begin{align*}
		(2-\alpha-\ve)\int_{\B}e^ue^{2\alpha u_k}\vp^2dx\leq C,
		\end{align*}
		provided $\int_{\Omega}e^{2(\alpha+\ve)u}dx<\infty.$ Now choosing $\alpha\in(0,\frac12)$ and $0<\ve<\frac 12-\alpha$ in the above relation, and then taking $k\to\infty$ we get that $e^{(1+2\alpha)u}\in L^1_{\mathrm{loc}}(\Omega)$. By an iteration argument we conclude that $e^{(1+2\alpha)u}\in L^1_{\mathrm{loc}}(\Omega)$ for every $\alpha\in(0,2)$.
		
		Next, sending $k\to\infty$ in \eqref{9} we see that
		$$\int_{\R}t^{1-2s}e^{2\alpha\ou}\Phi^2|\nabla\ou|^2dxdt<\infty.$$ Therefore, we can take limit $k\to\infty$ in \eqref{7} and \eqref{8}. Then the lemma would follow immediately as the second term on the right hand side of \eqref{7} (after taking $k\to\infty$) can be written as
		\begin{align*}
		\int_{\R}t^{1-2s}e^{2\alpha\ou}\nabla \ou\cdot\nabla\Phi^2dxdt
		&=\frac{1}{2\alpha} \int_{\R}t^{1-2s}\nabla e^{2\alpha\ou}\cdot\nabla \Phi^2dxdt\\
		&=-\frac{1}{2\alpha} \int_{\R} e^{2\alpha\ou}\nabla\cdot[t^{1-2s}\nabla \Phi^2]dxdt.
		\end{align*} Again, the boundary integral is zero as $\eta=1$ on $[0,1]$.
	\end{proof}

	\begin{lemma}
		\label{lem-decay} Let $\alpha>0$  and $\ve>0$ be such that $t^{1-2s}e^{2(\alpha+\ve) \ou}\in L^1_{\mathrm{loc}}(\overline{\R})$. Let $\Phi\geq 0$ be  a smooth function with compact support in $\overline{\R}$.  Assume that
		$$\int_{\R}t^{1-2s}e^{2\alpha\ou_k}\Phi|\nabla \ou_k|dxdt \to\infty. $$
		Then there exists a sub-sequence $\{k'\}\subset\{k\}$ such that  $$\int_{\R}t^{1-2s}e^{2\alpha\ou_{k'}}\Phi|\nabla \ou|dxdt
		=o(1)\int_{\R}t^{1-2s}e^{2\alpha\ou_{k'}}\Phi^2|\nabla \ou_{k'}|^2dxdt. $$ \end{lemma}
	
	\begin{proof} We shall use the trivial facts $|\nabla\ou_k|\leq|\nabla\ou|$ on $\R$ and  $|\nabla\ou_k|=|\nabla\ou|$ on $\{\ou<k\}$. By H\"older inequality with respect to the measure $t^{1-2s}e^{2\alpha \ou_k}dxdt$, we get that
		\begin{equation}
		\label{10-1}
		\begin{aligned}
		\int_{\{\ou<k\}}t^{1-2s}e^{2\alpha\ou_k}\Phi|\nabla\ou|dxdt &\leq C
		\left(\int_{\R}t^{1-2s}e^{2\alpha\ou_k}\Phi^2|\nabla\ou_k|^2dxdt\right)^\frac12
		\\&=o(1)\int_{\R}t^{1-2s}e^{2\alpha\ou_k}\Phi^2|\nabla\ou_k|^2dxdt.
		\end{aligned}
		\end{equation}
		Therefore, if the lemma were false, we would have
		\begin{equation}
		\label{cont}
		\begin{aligned}
		&\int_{\R}t^{1-2s}e^{2\alpha\ou_k}\Phi^2|\nabla\ou_k| ^2dxdt\\
		&\leq Ce^{2\alpha k}\int_{\{\ou\geq k\}}t^{1-2s}\Phi|\nabla \ou| dxdt\\
		&\leq Ce^{(\alpha-\ve) k}\left( \int_{\{\ou\geq k\}}t^{1-2s}\Phi^2|\nabla \ou| ^2dxdt\right)^\frac12 \left(\int_{\{\ou\geq k\}} e^{2(\alpha +\ve)\ou}t^{1-2s}  dxdt\right)^\frac12\\
		&=o(1)  e^{(\alpha-\ve) k}\left( \int_{\{\ou\geq k\}}t^{1-2s}\Phi^2 |\nabla \ou| ^2dxdt\right)^\frac12.
		\end{aligned}
		\end{equation}
		In particular, setting $d\mu=t^{1-2s}\Phi^2|\nabla\ou|^2dxdt$, we get \begin{align}
		\label{10}
		e^{2\alpha k}\mu(\{k-1\leq \ou< k\})=o(1)e^{(\alpha-\ve)k}\mu(\{\ou\geq k\})^\frac12,
		\end{align}
		which gives
		\begin{align*}
		\mu(\{k-1\leq \ou\leq k\})=o(1)e^{-(\alpha+\ve) k},
		\end{align*}
		where in the last equality we have used that $\mu(\{u\geq k\})=o(1)$.
		
		Next, we claim that if
		\begin{align*}
		\mu(\{k-1\leq \ou<k\})=o(1)e^{-\beta k},
		\end{align*}
		for some $\beta>0$ then
		\begin{align*}
		\mu(\{\ou\geq k\})=o(1)e^{-\beta k}\quad\text{and }
		\mu(\{k-1\leq \ou<k\})=o(1)e^{-(\alpha+\ve+ \frac\beta2)k}.
		\end{align*}
		The first part of the claim follows from
		$$\mu(\{\ou\geq k\})=\sum_{\ell=1}^\infty\mu(\{k+\ell-1\leq \ou< k+\ell\}) =\sum_{\ell=1}^\infty o(1)e^{-\beta(k+\ell)}=o(1)e^{-\beta k},$$
		and the second part follows immediately from \eqref{10}.
		
		Since the hypothesis of the above claim holds with $\beta=\alpha+\ve$, a repeated use of it gives
		$$\mu(\{\ou\geq k\})=o(1)e^{-(\alpha+\ve)k\sum_{i=1}^N\frac{1}{i^2}},$$
		for every $N\geq 1$. Then we could choose a large number $N$ such that
		$$\mu(\{\ou\geq k\})=o(1)e^{-(2\alpha+\ve)k}.$$
		Going back to \eqref{cont}, we derive that
		$$ \int_{\R}t^{1-2s}e^{2\alpha\ou_k}\Phi^2|\nabla \ou_k| ^2dxdt=o(1),$$
		which contradicts to \eqref{10-1} and $\int_{\R}t^{1-2s}e^{2\alpha\ou_k}\Phi|\nabla\ou_k|dxdt\to+\infty.$ Hence we finish the proof.
	\end{proof}
	
	Now we are in a position to prove Proposition \ref{prop-2.6}. 
	
	\begin{proof}[Proof of Proposition \ref{prop-2.6}]
		It follows from Lemmas \ref{lem-2.7} and \ref{lem-0.2} that $e^{pu}\in L^1_{\mathrm{loc}}(\B\setminus B_R)$ for every $p\in [1,5)$. To prove \eqref{est-p-1} we first show it holds for $\alpha\in(0,\frac12)$. Indeed, by H\"older inequality and Lemma \ref{leh.1} we get that for $\alpha\in(0,\frac12)$
		\begin{equation*}
		\int_{B_{2r}\setminus B_r}e^{2\alpha u}dx\leq \left(\int_{B_{2r}\setminus B_r}e^udx\right)^{2\alpha}(\int_{B_{2r}\setminus B_r}1dx)^{1-2\alpha}\leq Cr^{n-4\alpha s}.
		\end{equation*}
		Next we claim that if
		\begin{align}
		\label{claim-2.31}
		\int_{B_{2r}\setminus B_{r}}e^{2\alpha u}dx\leq Cr^{n-4\alpha s}\quad \text{for  }r>2R,
		\end{align}
		for some $\alpha\in (0,\min\{\frac{n}{4s},2\})$, then
		\begin{align}
		\label{claim-2.32}\int_{B_{2r}\setminus B_{r}}e^{(1+2\alpha) u}dx\leq Cr^{n-2(1+2\alpha) s}\quad\text{for }r>3R.
		\end{align}
		By \eqref{claim-2.31}, it is not difficult to show that
		\begin{equation}
		\label{3.claim}
		\int_{B_r\setminus B_{2R}}e^{2\alpha u}dx\leq Cr^{n-4\alpha s}\quad \text{for  }r> 2R.
		\end{equation}
		Indeed, for  $r>2R$ of the form $r=2^{N_1}$ with some positive integer $N_1$,  and taking  $N_2$ to be the smallest integer such that $2^{N_2}\geq 2R$,  by  \eqref{claim-2.31} we deduce 
		\begin{equation}
		\label{3.claim1}
		\begin{aligned}
		\int_{B_{2^{N_1}}\setminus B_{2R}}e^{2\alpha u}dx=
		&\int_{B_{2^{N_2}}\setminus B_{2R}}e^{2\alpha u}dx+\sum_{\ell=1}^{N_1-N_2}\int_{B_{2^{N_2+\ell}}\setminus B_{2^{N_1+\ell}}}e^{2\alpha u}dx\\
		\leq & ~C+ C\sum_{\ell=1}^{N_1-N_2}(2^{N_2+\ell})^{n-4\alpha s}\\
		\leq & ~C2^{N_1(n-4\alpha s)},
		\end{aligned}
		\end{equation}
		where we used $n-4\alpha s>0.$ Then using the hypothesis \eqref{claim-2.31},   we derive the following decay estimate
		\begin{equation}
		\label{h.11}
		\begin{aligned}
		\int_{|y|\geq r}\frac{e^{2\alpha u(y)}}{|y|^{n+2s}}dy=~&
		\sum_{i=0}^\infty\int_{2^{i+1}r\geq |y|\geq 2^ir}\frac{e^{2\alpha u(y)}}{|y|^{n+2s}}dy\\
		\leq~&\frac{ C}{r^{2s+4\alpha s}}\sum_{i=0}^\infty\frac{1}{2^{(2+4\alpha)si}}
		\leq Cr^{-2s-4\alpha s}.
		\end{aligned}
		\end{equation}
		On the other hand, by \eqref{Poisson-repre} we get for $|x|\geq \frac32R$ that \begin{align*}
		\ou(x,t) &\leq C\frac{t^{2s}}{(R+t)^{n+2s}}\int_{|y|\leq 2R}u^+(y)dy+\int_{\B}\chi_{\B\setminus B_{2R}}(y)u(y)P(X,y)dy\\
		&\leq C+\int_{\B}\chi_{\B\setminus B_{2R}}(y)u(y)P(X,y)dy,
		\end{align*}
		where $\chi_A$ denotes   the characteristic function of a set $A$. Using Jensen's inequality
		\begin{align}
		\label{jensen-1}
		e^{2\alpha\ou(x,t)}\leq C\int_{\B}\left(e^{2\alpha u(y)} \chi_{\B\setminus B_{2R}}(y)+\chi_{B_{2R} }(y) \right)P(X,y)dy \quad\text{for }|x|\geq 3R.
		\end{align}
		For $r>3R$ and $t\in(0,r)$,  integrating both sides of the inequality \eqref{jensen-1} on $B_{r}\setminus B_{3R}$ with respect to $x$
		\begin{align*}
		\int_{B_{r}\setminus B_{3R}}e^{2\alpha\ou(x,t)}dx
		&\leq C \int_{|y|\leq 2R}\int_{|x|\leq r}P(X,y)dxdy+C\int_{|y|\geq 2r}\int_{|x|\leq r}e^{2\alpha u(y)}P(X,y)dxdy\\
		&\quad+C\int_{2R\leq |y|\leq 2r}e^{2\alpha u(y)}\int_{|x|\leq r}P(X,y)dxdy \\ &\leq CR^n +Cr^{n+2s}\int_{|y|\geq 2r}\frac{e^{2\alpha u(y)}}{|y|^{n+2s}}dy+C\int_{2R\leq |y|\leq 2r}e^{2\alpha u(y)}dy\\
		&\leq C+Cr^{n-4\alpha s}\leq Cr^{n-4\alpha s},
		\end{align*}
		where we used \eqref{claim-2.31}, \eqref{3.claim1}, and \eqref{h.11}.  Now we fix non-negative smooth functions $\vp$ on $\B$ and $\eta$ on $ [0,\infty)$ such that
		\begin{align*}
		\vp(x)=\left\{\begin{array}{ll}
		1 &\quad\text{on } B_{2}\setminus B_1\\
		\\
		0&\quad\text{on }B_{2/3}\cup B_{3}^c
		\end{array}\right.,\quad
		\eta(t)=\left\{\begin{array}{ll}
		1 &\quad\text{on }[0,1]\\
		\\
		0&\quad\text{on } [2,\infty)
		\end{array}\right..
		\end{align*}
		For $r>0$ we set $\Phi_r(x,t)=\vp(\frac xr)\eta(\frac tr)$. Then  $\Phi_r$ is a good test function in Lemma \ref{lem-0.2} for $r\geq 3R$. Therefore, as  $|\nabla \Phi_r |\leq  \frac Cr$, we obtain
		\begin{equation*}
		\begin{aligned}
		\int_{\R}t^{1-2s}e^{2\alpha\ou}|\nabla  \Phi_r |^2dxdt
		&\leq Cr^{-2}\int_0^{2r} t^{1-2s}\int_{B_{3r}\setminus B_{2r/3}}e^{2\alpha\ou(x,t)}dxdt\\
		&\leq Cr^{n-2-4\alpha s}\int_0^{2r}t^{1-2s} dt\\
		&\leq Cr^{n-2s(1+2\alpha)}.
		\end{aligned}
		\end{equation*}
		In a similar way we have
		\begin{equation*}
		\left|\int_{\R}e^{2\alpha\ou} \nabla\cdot [t^{1-2s}\nabla  \Phi_r^2]dxdt\right|\leq Cr^{n-2s(1+2\alpha)}.
		\end{equation*}
		Then \eqref{claim-2.32} follows from \eqref{h.6} of Lemma \ref{lem-0.2}. Thus, we prove the claim. Repeating the above arguments finitely many times we get \eqref{est-p-1}, while \eqref{est-p-2} follows immediately as $B_{|x|/2}(x)\subset B_{2r}\setminus B_{r/2}$ with $r=|x|$.
		
		In the spirit of the estimate   \eqref{3.claim1},  one can obtain  the second conclusion \eqref{est-p-3} of Proposition \ref{prop-2.6}.  
	\end{proof}

	We end this section by proving Theorem \ref{th1.2}. 
	\begin{proof}[Proof of Theorem \ref{th1.2}.]
Since  $u$ is stable in $\Omega$, following the proof of  Proposition \ref{prop-2.6}, we get that  $e^{u}\in L^p_{\mathrm{loc}}(\Omega)$ for every $p\in[1,5)$. We write $u=v+h$ where $$v(x):=c({n,s})\int_{\Omega}\frac{1}{|x-y|^{n-2s}}e^{u(y)}dy.$$ Then $(-\D)^sh=0$ in $\Omega$, and hence, $h$ is smooth in $\Omega$. For $n<10s $ we can choose $p\in (1,5)$ such that $p'(n-2s)<n$. Then using H\"older inequality and the fact that   $e^{p u}\in L^1_{\mathrm{loc}}(\Omega)$, we have that $ v\in L^\infty_{\mathrm{loc}}(\Omega)$. Thus, $ u \in L^\infty_{loc}(\Omega)$.  Smoothness of $u$ follows by the standard bootstrapping argument.	
	\end{proof}

	\medskip
	\section{proof of Theorem \ref{th1.1}}
	\subsection{Non-existence of stable homogeneous solution}
	In this subsection, we shall prove the non-existence of stable homogeneous solutions.
	\begin{theorem}
		\label{th3.1}
		There is no stable solution of \eqref{fg-1} of the form $\tau(\theta)-2s\log r$ provided
		\begin{equation}
		\label{3.1}
		\dfrac{\Gamma(\frac n2)\Gamma(1+s)}{\Gamma(\frac{n-2s}{2})}> \dfrac{\Gamma^2(\frac{n+2s}{4})}{\Gamma^2(\frac{n-2s}{4})}.
		\end{equation}
		Here $\theta=\frac{x}{|x|}\in\mathbb{S}^{n-1}.$
	\end{theorem}
	
	\begin{proof}
		It is easy to see that for any radially symmetric function $\vp\in C_c^\infty(\B)$
		\begin{equation*}
		\int_{\B}\tau(\theta)\s\vp dx=0.
		\end{equation*}
		Then  from \eqref{weak-fg}
		\begin{equation*}
		\begin{aligned}
		\int_{\B}e^{\tau(\theta)-2s\log|x|}\vp dx
		=~&\int_{\B}(\tau(\theta)-2s\log|x|)\s\vp dx\\
		=~&\int_{\B}(-2s\log|x|)\s\vp dx\\=~&A_{n,s}\int_{\B}\frac{\vp}{|x|^{2s}}dx,
		\end{aligned}
		\end{equation*}
		where  the last equality follows from  (see e.g. \cite{r1})
		$$\s\log\frac{1}{|x|^{2s}}=A_{n,s}\frac{1}{|x|^{2s}}=2^{2s}\frac{\Gamma(\frac n2)\Gamma(1+s)}{\Gamma(\frac{n-2s}{2})}\frac{1}{|x|^{2s}}.$$
		This leads to 
		\begin{equation*}
		0=\int_{\B}(e^{\tau(\theta)}-A_{n,s})\frac{\vp}{|x|^{2s}}dx=\int_0^\infty r^{n-1-2s}\vp(r)\int_{\mathbb{S}^{n-1}}(e^{\tau(\theta)}-A_{n,s})d\theta dr,
		\end{equation*}
		which implies
		\begin{equation}
		\label{3.5}
		\int_{\mathbb{S}^{n-1}}e^{\tau(\theta)}d\theta=A_{n,s}|\mathbb{S}^{n-1}|.
		\end{equation}

		Now we shall use the stability  condition to derive a counterpart equation of \eqref{3.5}.  We fix a  radially symmetric smooth   cut-off function 
		\begin{equation*}
		\eta(x)=\begin{cases}
		1\quad  &\mathrm{for}~|x|\leq1,\\ \rule{0cm}{.5cm}
		0\quad  &\mathrm{for}~|x|\geq2,
		\end{cases}
		\end{equation*}
		and set	
		$$\eta_{\varepsilon}(x)=\left(1-\eta\left(\frac{2x}{\e}\right)\right)\eta(\e x).$$ 
		It is easy to see that $\eta_{\e}=1$ for $\e<r<\e^{-1}$ and $\eta_{\e}=0$ for either $r<\frac{\e}{2}$ or $r>\frac{2}{\e}$. We test the stability condition \eqref{1.stablecondition} on the function $\psi(x)=r^{-\frac{n-2s}{2}}\eta_{\e}(r)$. Let $|y|=rt$ and we notice that
		\begin{equation*}
		\begin{aligned}
		\int_{\B}\frac{\psi(x)-\psi(y)}{|x-y|^{n+2s}}dy
		&= r^{-\frac n2-s}\int_0^\infty\int_{\S}\frac{\eta_{\e}(r)-t^{-\frac{n-2s}{2}}\eta_{\e}(rt)}{(t^2+1-2t\langle \theta,\omega\rangle)^{\frac{n+2s}{2}}}t^{n-1}dtd\omega\\
		&= r^{-\frac n2-s}\eta_{\e}(r)\int_0^{\infty}\int_{\S}\frac{1-t^{-\frac{n-2s}{2}}}
		{(t^2+1-2t\langle\theta,\omega\rangle)^{\frac{n+2s}{2}}}t^{n-1}dtd\omega\\
		&\quad +r^{-\frac n2-s}\int_0^{\infty}\int_{\S}\frac{t^{n-1-\frac{n-2s}{2}}(\eta_{\e}(r)-\eta_{\e}(rt))}
		{(t^2+1-2t\langle\theta,\omega\rangle)^{\frac{n+2s}{2}}}dtd\omega.
		\end{aligned}
		\end{equation*}
		It is known that (see e.g.  \cite[Lemma 4.1]{fa})
		\begin{equation*}
		\Lambda_{n,s}=c_{n,s}\int_0^{\infty}\int_{\S}\frac{1-t^{-\frac{n-2s}{2}}}
		{(t^2+1-2t\langle\theta,\omega\rangle)^{\frac{n+2s}{2}}}t^{n-1}dtd\omega.
		\end{equation*}
		Therefore,
		\begin{equation*}
		\begin{aligned}
		c_{n,s}\int_{\B}\frac{\psi(x)-\psi(y)}{|x-y|^{n+2s}}dy
		=~&c_{n,s}r^{-\frac n2-s}\int_0^{\infty}\int_{\S}\frac{t^{n-1-\frac{n-2s}{2}}(\eta_{\e}(r)-\eta_{\e}(rt))}
		{(t^2+1-2t\langle\theta,\omega\rangle)^{\frac{n+2s}{2}}}dtd\omega\\
		&+\Lambda_{n,s}r^{-\frac n2-s}\eta_{\e}(r).
		\end{aligned}
		\end{equation*}
		Based on the above computations, we compute the left hand side of the stability inequality \eqref{1.stablecondition},
		\begin{equation}
		\label{3.8}
		\begin{aligned}
		&\frac{c_{n,s}}{2}\int_{\B}\int_{\B}\frac{(\psi(x)-\psi(y))^2}{|x-y|^{n+2s}}dxdy\\
		&=c_{n,s}\int_{\B}\int_{\B}\frac{(\psi(x)-\psi(y))\psi(x)}{|x-y|^{n+2s}}dxy\\
		&=\Lambda_{n,s}|\S|\int_0^\infty r^{-1}\eta_{\e}^2(r)dr\\
		&\quad+c_{n,s}\int_0^\infty\left[\int_0^\infty r^{-1}\eta_{\e}(r)(\eta_{\e}(r)-\eta_{\e}(rt))dr\right]\\
		&\quad\quad\times\int_{\S}\int_{\S}\frac{t^{n-1-\frac{n-2s}{2}}}
		{(t^2+1-2t\langle\theta,\omega\rangle)^{\frac{n+2s}{2}}}d\omega d\theta dt.
		\end{aligned}
		\end{equation}
		We compute the right hand side of  the stability inequality \eqref{1.stablecondition} for the test function $\psi(x)=r^{-\frac{n}{2}+s}\eta_{\e}(r)$ and $u(r)=-2s\log r+\tau(\theta)$,
		\begin{equation}
		\label{3.9}
		\begin{aligned}
		\int_{\B}e^u\psi^2=&\int_0^\infty\int_{\S}\eta_{\e}^2(r)r^{-2s}r^{-(n-2s)}e^{\tau(\theta)}r^{n-1}drd\theta\\
		=&\int_0^\infty r^{-1}\eta^2_{\e}(r)dr\int_{\S}e^{\tau(\theta)}d\theta.
		\end{aligned}
		\end{equation}
		From the definition of the function $\eta_{\e}$, we have
		\begin{equation*}
		\int_0^\infty r^{-1}\eta_{\e}^2(r)dr=\log\frac{2}{\e}+O(1).
		\end{equation*}
		One can see that both the first term on the right hand side of \eqref{3.8} and the right hand side of \eqref{3.9} carry the term $\int_0^\infty r^{-1}\eta_{\e}^2(r)dr$ and it tends to $\infty$ as $\e\to0$. Next we claim that
		\begin{equation}
		\label{3.10}
		f_{\e}(t):=\int_0^\infty r^{-1}\eta_{\e}(r)(\eta_{\e}(r)-\eta_{\e}(rt))dr =O(\log t).
		\end{equation}
		From the definition of $\eta_{\e},$ we have
		\begin{equation*}
		f_{\e}(t)=\int_{\frac{\e}{2}}^{\frac{2}{\e}}r^{-1}\eta_{\e}(r)(\eta_{\e}(r)-\eta_{\e}(rt))dr.
		\end{equation*}
		Notice that
		\begin{equation*}
		\eta_{\e}(rt)=\begin{cases}
		1, \quad &\mathrm{for}~\frac{\e}{t}<r<\frac{1}{t\e},\\
		\\
		0, \quad &\mathrm{for~either}~r<\frac{\e}{2t}~\mathrm{or}~r>\frac{2}{t\e}.
		\end{cases}
		\end{equation*}
		Now we consider various ranges of value of $t\in(0,\infty)$ to establish the claim \eqref{3.10}.
		\begin{equation*}
		f_{\e}(t)\approx\begin{cases}
		-\int_{\frac{\e}{2}}^{\frac{2}{t\e}}r^{-1}dr+\int_{\frac{\e}{2}}^{\frac{2}{\e}}r^{-1}dr\approx \log\e=O(\log t),\quad~&\mathrm{if}~\frac{1}{t\e}<\e,\\
		\\
		-\int_{\frac{\e}{2}}^{\e}r^{-1}dr+\int_{\frac{1}{\e t}}^{\frac{2}{\e}}r^{-1}dr\approx \log t, \quad &\mathrm{if}~\frac{\e}{t}<\e<\frac{1}{\e t},\\
		\\
		\int_{\frac{\e}{2}}^{\frac{\e}{t}}r^{-1}dr-\int_{\frac{1}{\e}}^{\frac{2}{\e }}r^{-1}dr\approx \log t,\quad & \mathrm{if}~\e<\frac{\e}{t}<\frac{1}{\e},\\
		\\
		\int_{\frac{\e}{2}}^{\frac{2}{\e}}r^{-1}dr-\int_{\frac{\e}{2t}}^{\frac{2\e}{t}}r^{-1}dr\approx \log\e=O(\log t),
		\quad &\mathrm{if}~\frac{1}{\e}<\frac{\e}{t}.
		\end{cases}
		\end{equation*}
		The other cases can be treated similarly. From this one can see that
		\begin{equation*}
		\begin{aligned}
		&\int_0^{\infty}\left[\int_0^\infty r^{-1}\eta_{\e}(r)(\eta_{\e}(r)-\eta_{\e}(rt))\right]\int_{\S}\int_{\S}
		\dfrac{t^{n-1-\frac{n-2s}{2}}}
		{(t^2+1-2t\langle\theta,\omega\rangle)^{\frac{n+2s}{2}}}d\omega d\theta dt\\
		&\approx \int_0^\infty\int_{\S}\int_{\S}\dfrac{t^{n-1-\frac{n-2s}{2}}\log t}
		{(t^2+1-2t\langle\theta,\omega\rangle)^{\frac{n+2s}{2}}}d\omega d\theta dt\\
		& = O(1)
		\end{aligned}
		\end{equation*}
		Collecting the higher order term $(\log\e)$, we get
		\begin{equation}
		\label{3.12}
		\Lambda_{n,s}|\S|\geq \int_{\S}e^{\tau(\theta)}d\theta.
		\end{equation}
		From \eqref{3.5} and \eqref{3.12}, we obtain that
		\begin{equation*}
		\Lambda_{n,s}\geq A_{n,s},
		\end{equation*}
		which contradicts to the assumption \eqref{3.1}. Therefore, such homogeneous solution does not exist and we finish the proof.
	\end{proof}

	\subsection{Monotonicity formula}
	In this subsection  we  give a proof of Theorem \ref{th4.2}, and estimate the   terms appearing in \eqref{1.monotonicity}.
	
	\begin{proof}[Proof of Theorem \ref{th4.2}.]
We prove the theorem  for $u$  sufficiently smooth, and give the necessary details for the general case in  Remark \ref{re1.1} below.  Without loss of generality, we may assume that $x_0=0$ and the balls $B_\lambda^{n+1}$ are centered at $0$. Set,
		\begin{equation*}
		E_1(\ou,\lambda)=\lambda^{2s-n}\left(\int_{B^{n+1}_\lambda\cap\R}t^{1-2s}\frac{|\nabla\ou|^2}{2}dxdt-\kappa_s\int_{B_{\lambda}^{n+1}\cap\partial\R}e^{\ou}dx\right).
		\end{equation*}
		Define
		\begin{equation*}
		\ou^\lambda (X)=\ou(\lambda X)+2s\log\lambda.
		\end{equation*}
		Then
		\begin{equation}
		\label{3.relation}
		E_1(\ou,\lambda)=E_1(\ou^{\lambda},1).
		\end{equation}
		Differentiating $\ou^\l$ with respect to $\lambda $, we have
		\begin{equation*}
		\lambda\partial_{\lambda}\ou^{\lambda}=r\partial_r \ou^{\lambda}+2s.
		\end{equation*}
		Differentiating the right hand side of \eqref{3.relation}, we find
		\begin{equation}
		\label{3.diff}
		\begin{aligned}
		\frac{d E_1}{d\lambda}(\ou,\lambda)=~&\int_{B^{n+1}_1\cap\R}t^{1-2s}\nabla \ou^{\lambda}\nabla \partial_{\lambda}\ou^{\lambda}dxdt
		-\kappa_s\int_{B_1^{n+1}\cap\partial\R}e^{\ou^{\lambda}}\partial_{\lambda}\ou^{\lambda}dx\\
		=~&\int_{\partial B_1^{n+1}\cap\R}t^{1-2s}\partial_r\ou^\lambda \partial_\lambda \ou^\lambda d\sigma\\
		=~&\lambda \int_{\partial B_1^{n+1}\cap\R}t^{1-2s}(\partial_\lambda \ou^\lambda)^2d\sigma
		-2s\int_{\partial B_1^{n+1}\cap\R}t^{1-2s}\partial_\lambda \ou^\lambda d\sigma.
		\end{aligned}
		\end{equation}
		We notice that
		\begin{equation}
		\label{3.m2}
		E(\ou,0,\lambda)=E(\ou^\lambda,0,1)=E_1(\ou^\lambda,1)+2s\int_{\partial B_1^{n+1}\cap\R}t^{1-2s}\ou^\lambda d\sigma.
		\end{equation}
		From \eqref{3.diff} and \eqref{3.m2}, we get
		\begin{equation*}
		\begin{aligned}
		\frac{d E}{d\lambda}(\ou,0,\lambda)=~&\lambda \int_{\partial B_1^{n+1}\cap\R}t^{1-2s}(\partial_\lambda \ou^\lambda)^2d\sigma\\
		=~&\lambda^{2s-n}\int_{\partial B_\lambda^{n+1}\cap\R}t^{1-2s}(\partial_r\ou+\frac{2s}{r})^2d\sigma.
		\end{aligned}
		\end{equation*}
		Hence, we finish the proof.
	\end{proof}

	\begin{remark}
	\label{re1.1} 
Setting     $u^\lambda_{\e}:=u^\lambda*\rho_{\e}$   ($(\rho_{\e})_{\e>0}$ are the standard mollifiers) we  see that $u^\lambda_{\e}$ satisfies $\s u^\lambda_{\e}=e^{u^\lambda_\e}*\rho_{\e}$. Then we consider $E(\ou^\lambda_{\e},0,1)$. Following the same computations as before, we could get
    \begin{equation*}
    \frac{d}{d\lambda}E(\ou^\lambda_{\e},0,1)=\l\int_{\partial B^{n+1}_1\cap\R}t^{1-2s}\left(\partial_\l\ou^\l_\e\right)^2d\sigma+\kappa_s\int_{B_1}(e^{u^\l}*\rho_{\e}-e^{u_\e^\l})\partial_\l u_\e^\l dx,
    \end{equation*}
    where the second term on the right hand side could be controlled by assuming $u\in W^{1,2}_{\mathrm{loc}}(\B)$ and $e^{u}\in L^2_{\mathrm{loc}}(\B)$, and it converges to zero as $\e$ tends  to $0$. 
    
    We could also   assume  $u\in W^{1,q}_{\mathrm{loc}}(\B)$ with $q$ slightly bigger than $5/4$. In sacrifice of the less regularity assumption, we have to consider the monotonicity formula by truncating the region where $u$ is unstable, i.e., 
    \begin{equation*}
    \begin{aligned}
    E_R(\ou^\l,0,1)=~&\left(\frac12\int_{(B_1^{n+1}\setminus B_{2R/\l}^{n+1})\cap \R}t^{1-2s}|\nabla\ou^\l|^2dxdt-\kappa_s\int_{B_1\setminus B_{2R/\l}}e^{\ou^\l}dx\right)\\
    &+2s\int_{\partial (B^{n+1}_1\setminus B^{n+1}_{2R/\l})\cap\R}t^{1-2s}\ou^\l d\sigma,
    \end{aligned}
    \end{equation*}
	where $R$ is chosen such that $u$ is stable in $\B\setminus B_R(0).$ For $E_{R}(\ou^\l,0,1)$, after  some   computations we could show that
\begin{equation}\label{4.10}
\frac{d}{d\lambda}E_R(\ou^\lambda,0,1)=\l\int_{\partial B^{n+1}_1\cap\R}t^{1-2s}\left(\partial_\l\ou^\l\right)^2d\sigma+o(\l^{-1-\delta}),~\ \mathrm{as}~\ \l\to\infty,
\end{equation} 
for some  $\delta>0$.  In this case we are not claiming that $E_R(\ou^\l,0,1) $ is monotone increasing with respect to $\lambda$. However, as $\delta>0$,  the above estimate is good enough for our purposes. To be precise, it will be used to show that the constant $c_\l$ has a uniform lower bound for $\l\geq 1$  (see Proposition \ref{prf.1}), and that the limit function $\ou^\infty$ is homogeneous (see subsection \ref{section5}).  
\end{remark}

	By Lemma \ref{f.1} we derive the following expression for the third term in the monotonicity formula \eqref{1.monotonicity}.
	\begin{lemma}
		\label{lef.2}
		Let $\ou^\l$ be the $s$-harmonic extension of $u^\l$, then
		\begin{equation*}
		\int_{\partial B_1^{n+1}\cap\R}t^{1-2s}\ou^\l(X)d\sigma=c_sc_\l+O(1),
		\end{equation*}
		where $c_\l$ is defined in \eqref{f.11} and $c_s$ is a positive finite number given by
		$$c_s:=\int_{\partial B_1^{n+1}\cap\R}t^{1-2s}d\sigma.$$
	\end{lemma}
	
	\begin{proof}
		Using the Poisson formula we have
		\begin{equation*}
		\ou^\l(X)=\int_{\B}P(X,z)u^\l(z)dz=c_\l+\int_{\B}P(X,z)v^\l(z)dz.
		\end{equation*}
		It follows from \eqref{f.4} that
		\begin{equation*}
		\int_{\B\setminus B_2}P(X,z)|v^\l(z)|dz\leq C\quad\mbox{for}\quad |X|\leq 1.
		\end{equation*}
		Therefore
		\begin{align*}
		\int_{\partial B^{n+1}_1\cap\R}t^{1-2s}\ou^\l(X)d\sigma
		=c_sc_\l+O(1)+\int_{\partial B_1^{n+1}\cap\R}t^{1-2s}\int_{B_2}P(X,z)v^\l(z)dzd\sigma.
		\end{align*}
		We denote the last term in the above equation by $II$. To estimate the term $II$, we claim that
		\begin{equation}
		\label{f.14}
		\int_{\partial B^{n+1}_1\cap\R}\int_{B_4}t^{1-2s}P(X,z)\frac{1}{|y-z|^{n-2s}}dzd\sigma\leq C\quad \mbox{for every}~y\in\B.
		\end{equation}
		Indeed, for $x\neq y$ we set $r=\frac12|x-y|$. Then we have
		\begin{equation}
		\label{f.15}
		\begin{aligned}
		&\int_{B_4}\frac{t}{|(x-z,t)|^{n+2s}|y-z|^{n-2s}}dz\\
		&\leq\left(\int_{B(y,r)}+\int_{B_4\setminus  B(y,r)}\right)\frac{t}{|(x-z,t)|^{n+2s}|y-z|^{n-2s}}dz\\
		&\leq C\frac{t}{(r+t)^{n+2s}}
		\int_{B(y,r)}\frac{1}{|y-z|^{n-2s}}dz
		+\frac{1}{r^{n-2s}}\int_{B_4\setminus  B(y,r)}\frac{t}{|(x-z,t)|^{n+2s}}dz\\
		&\leq C\left(\frac{tr^{2s}}{(r+t)^{n+2s}}
		+\frac{t^{1-2s}}{r^{n-2s}}\right)\leq C\left(\frac{1}{r^{n-1}}+\frac{t^{1-2s}}{r^{n-2s}}\right).
		\end{aligned}
		\end{equation}
We use the stereo-graphic projection $(x,t)\to\xi$ from $\partial B_1^{n+1}\cap\R\to\mathbb{R}^n\setminus B_1$, i.e.,
		$$(x,t)\to\xi=\frac{x}{1-t}.$$
		Then $r=\frac12\left|\frac{2\xi}{1+|\xi|^2}-y\right|$ and it follows that
		\begin{equation*}
		\int_{\partial B_1^{n+1}\cap\R}\frac{1}{r^{n-1}}d\sigma
		\leq\int_{|\xi|\geq 1}\frac{1}{r^{n-1}}\frac{1}{(1+|\xi|^2)^n}d\xi\leq C,
		\end{equation*}
		and
		\begin{equation}
		\label{f.17}
		\begin{aligned}
		\int_{\partial B_1^{n+1}\cap\R}\frac{t^{1-2s}}{r^{n-2s}}d\sigma
		\leq \int_{|\xi|\geq1}\frac{t^{1-2s}}{r^{n-2s}}
		\frac{1}{(1+|\xi|^2)^n}d\xi\leq C.
		\end{aligned}
		\end{equation}
		From \eqref{f.15}-\eqref{f.17}, we proved \eqref{f.14}. As a consequence, we have
		\begin{equation*}
		\begin{aligned}
		|II|&\leq C+C\int_{\partial B^{n+1}_1\cap\R}t^{1-2s}\int_{|z|\leq 2}P(X,z)\int_{|y|\leq 4}\frac{e^{u^\l(y)}}{|z-y|^{n-2s}}dydzd\sigma\\
		&\leq C+C\int_{|y|\leq 4}e^{u^\l(y)}\int_{\partial B_{1}^{n+1}\cap\R}\int_{|z|\leq 2}t^{1-2s}P(X,z)\frac{1}{|z-y|^{n-2s}}dzd\sigma dy\\
		&\leq C+C\int_{|y|\leq 4}e^{u^\l(y)}dy\leq C,
		\end{aligned}
		\end{equation*}
		where we used \eqref{f.1} and \eqref{f.5}. Hence we finish the proof.
	\end{proof}
	
To estimate the first term in the monotonicity formula \eqref{1.monotonicity}  we need the following result:
	
	\begin{lemma}
		\label{lef.3}
		We have
		$$\int_{B_r}|(-\D)^\frac s2u^\lambda(x)|^2dx\leq Cr^{n-2s}\quad\text{for~ }r\geq1,~\ \lambda\geq1.$$
	\end{lemma}
	\begin{proof}
		By a scaling argument,  it suffices to prove the lemma for $\lambda=1$. It follows from \eqref{f.11} that
		$$(-\D)^\frac s2u(x)=C\int_{\B}\frac{1}{|x-y|^{n-s}}e^{u(y)}dy,$$
		in the sense of distribution. Let $R>0$ be such that $u$ is stable outside $B_R$. For $r\gg R$, we decompose $B_{r}=B_{4R}\cup (B_{r}\setminus B_{4R}).$  Since  $u\in\dot H_{\mathrm{loc}}^s(\B)$,  we get
		\begin{equation}
		\label{4.4.0}
		\int_{B_{4R}}|(-\D)^\frac{s}{2} u|^2dx\leq C(R).
		\end{equation}
		While for $4R<|x|<r$ we estimate
		\begin{equation}
		\label{4.4.1}
		\begin{aligned}
		|(-\D)^\frac s2u(x)| &\leq C\left(\int_{B_{2R}}    +\int_{B_{2r}\setminus B_{2R}}+\int_{\B\setminus B_{2r}}\right)
		\frac{1}{|x-y|^{n-s}}e^{u(y)}dy\\
		&\leq  \frac{C}{|x|^{n-s}}+C\int_{B_{2r}\setminus B_{2R}}\frac{1}{|x-y|^{n-s}}e^{u(y)}dy+ C\int_{\B\setminus B_{2r}}
		\frac{1}{|y|^{n-s}}e^{u(y)}dy \\
		&=:C\left(\frac{1}{|x|^{n-s}}+I_1(x)+I_2 \right) .
		\end{aligned}
		\end{equation}
		Using \eqref{h.7} we bound
		\begin{align*}
		I_2=\sum_{k=0}^\infty \int_{r2^k\leq|x|\leq r2^{k+1}}\frac{e^{u(y)}}{|y|^{n-s}}dy  \leq
		C\sum_{k=0}^\infty \frac{(r2^{k+1})^{n-2s}}{(2^kr)^{n-s}}\leq \frac{C}{r^{s}}.
		\end{align*}
		Therefore,
		\begin{equation}
		\label{4.4.2}
		\int_{B_r}I_2^2dx\leq Cr^{n-2s}.
		\end{equation}
		For the second term $I_1(x)$, if $4s\geq n>2s$, by Lemma \ref{lemma-3.2} we have
		\begin{equation*}
		I_1(x)\leq C\int_{2R\leq|y|\leq 2r}\frac{1}{|x-y|^{n-s}|y|^{2s}}dy\leq \frac{C}{|x|^s},
		\end{equation*}
		and hence
		\begin{equation}
		\label{4.4.3}
		\int_{B_r\setminus B_{4R}}I_1^2(x)dx\leq Cr^s\int_{B_r\setminus B_{4R}}\int_{B_{2r}\setminus B_{2R}}\frac{1}{|x-y|^{n-s}}e^{2u(y)}dydx\leq Cr^{n-2s}.
		\end{equation}
		If $n>4s$,  by H\"older inequality with respect to the measure $\frac{dy}{|x-y|^{n-s}}$  we get
		\begin{align*}
		\left(\int_{B_{2r}\setminus B_{2R}}\frac{e^{u(y)}}{|x-y|^{n-s}}dy\right)^2
		\leq~&\int_{B_{2r}}\frac{dy}{|x-y|^{n-s}}\left(\int_{B_{2r}\setminus B_{2R}}\frac{e^{2u(y)}}{|x-y|^{n-s}}dy\right)\\
		\leq~& Cr^s\int_{B_{2r}\setminus B_{2R}}\frac{e^{2u(y)}}{|x-y|^{n-s}}dy.
		\end{align*}
		Hence by Proposition \ref{prop-2.6}, we get  that \eqref{4.4.3} also holds for $n>4s$.
Combining  \eqref{4.4.1}, \eqref{4.4.2} and  \eqref{4.4.3}, we deduce 
		\begin{equation}
		\label{4.4.4}
		\int_{B_r\setminus B_{4R}}|(-\D)^\frac{s}{2} u|^2dx\leq Cr^{n-2s}.
		\end{equation}
		Then the lemma follows from \eqref{4.4.0}, \eqref{4.4.4} and $n>2s.$
	\end{proof}

	We use Lemma \ref{lef.3} to prove: 
	\begin{lemma} We have 
		\label{lef.4}
		\begin{equation}
		\label{f.24}
		\int_{B_r^{n+1}\cap \R}t^{1-2s}|\nabla \ou^\l(X)|^2dxdt\leq C(r),\quad \forall r>0,\,\l\geq1.
		\end{equation}
	\end{lemma}
	\begin{proof}
		We write
		$$u^\l=u_1^\l+u_2^\l,$$ where
		\begin{align*}
		u^\l_1(x)&=c(n,s)\int_{\B}\left(\frac{1}{|x-y|^{n-2s}}-\frac{1}{(1+|y|)^{n-2s}}\right)\varphi(y)e^{u^\l(y)}dy+c_\l,\\
		u^\l_2(x)&=c(n,s)\int_{\B}\left(\frac{1}{|x-y|^{n-2s}}-\frac{1}{(1+|y|)^{n-2s}}\right)(1-\varphi(y))e^{u^\l(y)}dy.\\
		\end{align*}
		Here $\varphi\in C_c^\infty(B_{4r})$  is such that $\varphi=1$ in $B_{2r}$. As in the proof of Lemma \ref{lef.3} one can show that    $$\int_{\R}t^{1-2s}|\nabla\ou^\l_1(x)|^2dxdt=\kappa_s\int_{\B}\left|\ss u^\l_1(x)\right|^2dx\leq C(r).$$
		Here and in the following  $\ou_i^\l$ denotes the $s$-harmonic extension of $u_i^\l,~i=1,2$ respectively. It remains to prove that
		\begin{equation}
		\label{f.26}
		\int_{B_r^{n+1}\cap\R}t^{1-2s}|\nabla\ou^\l_2(x)|^2dxdt\leq C(r)\quad \mbox{for every}~r\geq 1,~\lambda\geq1.
		\end{equation}
		Following the  arguments of Lemma \ref{lef.1}, one could verify that
		\begin{equation}
		\label{f.27}
		\|\nabla u_2^\l\|_{L^\infty(B_{3r/2})}\leq C(r),\quad
		\int_{\B}\frac{|u_{2}^\l(x)|}{1+|x|^{n+2s}}dx\leq C(r),
		\end{equation}
                 and consequently,  
		\begin{equation}
		\label{f.28}
		\|u_2^\l\|_{L^\infty(B_{3r/2})}\leq C(r).
		\end{equation}
		To prove \eqref{f.26}, we shall consider $\partial_t\ou^\l_2$ and $\nabla_x\ou^\l_2$ seperately. For the first term we notice that
		\begin{align*}
		\partial_t\ou^\l_2(X)=~&\partial_t(\ou^\l_2(x,t)-u^\l_2(x))=d_{n,s}\partial_t\int_{\B}\frac{t^{2s}}{|(x-y,t)|^{n+2s}}(u_2^\l(y)-u_2^\l(x))dy\\
		=~&d_{n,s}\int_{\B}\partial_t\left(\frac{t^{2s}}{|(x-y,t)|^{n+2s}}\right)(u_2^\l(y)-u_2^\l(x))dy,
		\end{align*}
		where we used
		$$d_{n,s}\int_{\B}\frac{t^{2s}}{|(x-y,t)|^{n+2s}}dy=1.$$
		By \eqref{f.27},  for $|x|\leq r$ it holds that
		\begin{equation}  
		\label{f.29}
		\begin{aligned}
		&\left|\int_{\B\setminus B_ {3r/2}}\partial_t\left(\frac{t^{2s}}{|(x-y,t)|^{n+2s}}\right) (u_2^\l(y)-u_2^\l(x)) dy\right|\\
		&\leq C(r)t^{2s-1}\int_{\B\setminus B_ {3r/2}}\frac{(|u_2^\l(y)|+1)}{1+|y|^{n+2s}}dy\leq C(r)t^{2s-1}.
		\end{aligned}
		\end{equation}
		Using \eqref{f.27}-\eqref{f.28}, we see that
		\begin{equation}
		\label{f.30}
		\begin{aligned}
		&\int_{B_r^{n+1}\cap\R}t^{1-2s}\left(\int_{B_{3r/2}}\partial_t\left(\frac{t^{2s}}{|(x-y,t)|^{n+2s}}\right)(u_2^\l(y)-u_2^\l(x))dy\right)^2dxdt\\
		&\leq C(r)\|\nabla u^\l\|_{L^\infty(B_{3r/2})}^2\int_{B_r^{n+1}\cap\R}t^{1-2s}
		\left(\int_{B_{3r/2}}\partial_t\left(\frac{t^{2s}}{|(x-y,t)|^{n+2s}}\right)|x-y|dy\right)^2dxdt\\
		&\leq C(r)\|\nabla u^\l\|_{L^\infty(B_{3r/2})}^2\int_{B_r^{n+1}\cap\R}\left(t^{1-2s}+t^{2s-1}\right)dxdt\leq C(r).
		\end{aligned}
		\end{equation}
		By \eqref{f.29}-\eqref{f.30} we get
		\begin{equation}
		\label{f.31}
		\begin{aligned}
		\int_{B_r^{n+1}}t^{1-2s}|\partial_t\ou^\l_2|^2dxdt
		\leq C(r)\int_{B_r^{n+1}}\left(t^{1-2s}+t^{2s-1}\right)dxdt\leq C(r).
		\end{aligned}
		\end{equation}
	    For the term $\nabla_x\ou_2^\l$,  in a similar way   we get
		\begin{equation}
		\label{f.35}
		\int_{B_r^{n+1}\cap\R}t^{1-2s}|\nabla_x\ou^\l_2|^2dxdt\leq C(r).
		\end{equation}
		Then \eqref{f.24} follows from \eqref{f.31} and \eqref{f.35}. Thus the lemma is proved.
	\end{proof}
	\medskip

	\begin{proposition}
		\label{prf.1}
We have $c_\lambda=O(1)$ for $\lambda\in [1,\infty)$. Moreover,  
		\begin{equation*}
		\lim_{\l\to+\infty}E(\ou,0,\l)=\lim_{\l\to\infty}E(\ou^\l,0,1)<+\infty.
		\end{equation*}
	\end{proposition}
	\begin{proof}By \eqref{f.1} and Lemma \ref{lef.4}, we have  that
	\begin{equation*}
	\frac12\int_{ B^{n+1}_1\cap\R}t^{1-2s}|\nabla\ou^\l|^2dxdt-\kappa_s\int_{B^{n+1}_1\cap\partial\mathbb{R}_+^{n+1}}e^{\ou^\l}dx\quad\mbox{is bounded in}~\l\in[1,\infty).
	\end{equation*}
	Using Theorem \ref{th4.2} and Lemma \ref{lef.2} we get
	\begin{equation*}
	E(\ou^\l,0,1)=c_sc_\l+O(1)\geq E(\ou,0,1)=c_sc_1+O(1),
	\end{equation*}
	which implies that
	\begin{equation}
	\label{f.1.2}
	c_\l ~\ \mbox{is bounded from below for}~\ \l\geq1.
	\end{equation}
	Here we notice that the  uniform lower bound on $c_\l$ can also be obtained using  the truncated energy functional as defined in Remark \ref{re1.1}, thanks to \eqref{4.10}. 
	
	By \eqref{f.5} we have $$u^\l=v^\l+c_\l\geq c_\l-C~\ \text{on}~\  B_1~\ \text{for}~\ \l\geq 1.$$ Hence, $$c_\l\leq C\quad\text{for }\l\geq 1,$$ thanks to \eqref{f.1}.  
	Thus we obtain  that $c_\l$ is bounded. Therefore, by Lemma \ref{lef.2}  
	\begin{equation}
	\label{f.1.5}
	\int_{\partial B^{n+1}_1\cap\R}t^{1-2s}\ou^\l(X)d\sigma=O(1).
	\end{equation}
We conclude the proof.  
 \end{proof}
	
\begin{lemma}\label{lem-H1bound} For every $r>0$ and $\lambda\geq 1$ we have $$\int_{B_r^{n+1}\cap\R}t^{1-2s}\left(|\ou^\l|^2+|\nabla \ou^\l|^2\right)dxdt\leq C(r).$$  \end{lemma}
\begin{proof} Based on Lemma \ref{lef.4}, we only need to show that $$\int_{B_r^{n+1}\cap\R}t^{1-2s}|\ou^\l|^2dxdt\leq C(r).$$ 
Together with $(w^+)^2\leq 2e^w$ and Jensen's inequality we get 
\begin{align*} \int_{B_r^{n+1}\cap\R}t^{1-2s}|(\ou^\l)^+|^2dxdt\leq C \int_{B_r^{n+1}\cap\R}t^{1-2s} \int_{\B}e^{u^\l(y)}P(X,y)dydxdt\leq C,
\end{align*}
where the last inequality follows from \eqref{f.1} and Lemma \ref{leh.3}.  As $c_\l=O(1)$, we see that $$u^\l(x)\geq -C\log(2+|x|)\quad\text{on}\quad \B,$$ thanks to \eqref{f.10}. In particular, $\ou^\l (X)\geq -C(r)$ for $|X|\leq r$. Then the lemma follows immediately. 
\end{proof}

	\medskip
	\subsection{Proof of Theorem \ref{th1.1}} \label{section5}
	In this subsection we provide the proof of Theorem \ref{th1.1}.
	
	\begin{proof}[Proof of Theorem \ref{th1.1}.]
	Let  $u$  be  finite Morse index solution to \eqref{fg-1} for some  $n>2s$ satisfying  \eqref{1.stable1}. Let $R>1$ be such  that $u$   is stable outside the ball  $B_R$. 
	\medskip
		
		From Lemma \ref{lem-H1bound} we obtain that there exists a sequence $\lambda_i\to+\infty$ such that $\ou^{\lambda_i}$ converges weakly in $\dot{H}^1_{\mathrm{loc}}(\overline\R,t^{1-2s}dxdt)$ to a function $\ou^\infty$. In addition, we have $\ou^{\l_i}\to \ou^\infty$ almost everywhere. To show that $\ou^\infty$ satisfies \eqref{weak-fg}, we need to verify two things. First, we need to show that for any $\ve$ there exists $r\gg1$ such that
		\begin{equation}
		\label{5.1}
		\int_{\B\setminus B_r}\frac{|u^\l|}{1+|y|^{n+2s}}dy<\ve,\quad\text{for every }\lambda\geq 1.
		\end{equation}
		Indeed,  as in the proof of  \eqref{f.111}, we get
		\begin{equation*}
		\begin{aligned}
		\int_{\B\setminus B_r}\dfrac{|u^\l(x)|}{1+|x|^{n+2s}}dx\leq~&
		C\int_{\B\setminus B_r}\dfrac{c_\lambda}{1+|x|^{n+2s}}dx
		+C\int_{\mathbb{R}^n\setminus B_{r_0}}e^{u^\l(y)}|E(y)|dy\\
		&+C\int_{\B\setminus B_r}\int_{B_{r_0}}
		\frac{e^{u^\l(y)}}{|x-y|^{n-2s}(1+|x|^{n+2s})}dydx\\
		&+C\int_{\B\setminus B_r}\int_{B_{r_0}}
		\frac{e^{u^\l(y)}}{((1+|y|)^{n-2s})(1+|x|^{n+2s})}dydx\\
		\leq ~& Cr^{-2s}+Cr_0^{-\gamma/2}+Cr^{-2s}r_0^{n-2s}.
		\end{aligned}
		\end{equation*}
		We could first choose $r_0$ large enough such that $Cr_0^{-\gamma/2}\leq\ve/2$ and then choose $r$ such that $Cr^{-2s}r_0^{n-2s}+Cr^{-2s}\leq\ve/2$. Thus, \eqref{5.1} is proved. As a consequence, we could show that $u^\infty\in L_s(\B)$, and for any $\vp\in C_c^\infty(\B)$
		\begin{equation}
		\label{5.2}
		\lim\limits_{i\to\infty}\int_{\B}u^{\lambda_i}\s\vp dx=\int_{\B}u^{\infty}\s\vp dx.
		\end{equation}
		The second point we need to prove is that $e^{u^{\lambda_i}}$ converge to $e^{u^{\infty}}$ in $L_{\mathrm{loc}}^1(\B).$ By \eqref{est-p-2} we can easily see that $e^{u^{\lambda_i}}$ is uniformly integrable in $L_{\mathrm{loc}}^1(\B\setminus\{0\})$. Using  \eqref{est-p-3}, around the origin we get  
		\begin{equation*}
		\int_{B_{\ve}}e^{u^{\lambda_i}}dx=\lambda_i^{2s-n}\int_{B_{\l_i\ve}}e^{u}dx\leq C\ve^{n-2s}.
		\end{equation*}
		Therefore, we have $(e^{u^{\lambda_i}})$ is uniformly integrable in $L_{\mathrm{loc}}^1(\B)$,  and together with $u^{\lambda_i} \to u^\infty$ a.e., we get for any $\vp\in C_c^\infty(\B)$
		\begin{equation}
		\label{5.3}
		\lim\limits_{i\to\infty}\int_{\B}e^{u^{\lambda_i}}\vp dx=\int_{\B}e^{u^{\infty}}\vp dx.
		\end{equation}
		Then $u^\infty$ satisfies equation \eqref{weak-fg} follows from \eqref{5.2} and \eqref{5.3}.
		
		
		Now  we show that the limit function $\ou^\infty$ is homogenous, and is of the form $-2s\log r+\tau(\theta)$.  Based on the above convergences, we get for any $r>0$, 
		\begin{equation}
		\label{5.r}
		{\lim_{i\to\infty}E(\ou,0,\l_i r)~\ \mbox{is independent of}~\ r}.
		\end{equation}
	    Indeed, for any two positive numbers $r_1 <r_2$ we have
	    \begin{equation*}
	    \lim_{i\to\infty}E(\ou,0,\l_i r_1)\leq\lim_{i\to\infty}E(\ou,0,\l_i r_2).
        \end{equation*}
        On the other hand, for any $\lambda_i$, we can choose $\lambda_{m_i}$ such that $\{\lambda_{m_i}\}\subset\{\lambda_i\}$ and $\lambda_i r_2\leq \lambda_{m_i} r_1$. As a consequence, we have
        \begin{equation*}
        \lim_{i\to\infty}E(\ou,0,\l_i r_2)\leq\lim_{i\to\infty}E(\ou,0,\l_{m_i} r_1)=\lim_{i\to\infty}E(\ou,0,\l_{i} r_1).
        \end{equation*}
This finishes the proof of \eqref{5.r}. Using \eqref{5.r} we see that for $R_2>R_1>0,$
		\begin{equation*}
		\begin{aligned}
		0=~&\lim_{i\to\infty}E(\ou,0,\l_i R_2)-\lim_{i\to+\infty}E(\ou,0,\l_i R_1)\\
		=~&\lim_{i\to\infty}E(\ou^{\l_i},0, R_2)-\lim_{i\to\infty}E(\ou^{\l_i},0,R_1)\\
		\geq~&\lim_{i\to\infty}\inf\int_{\left(B_{R_2}^{n+1}\setminus B_{R_1}^{n+1}\right)\cap\R}t^{1-2s}r^{2s-n}\left(\frac{\partial\ou^{\l_i}}{\partial r}+\frac{2s}{r}\right)^2dxdt\\
		\geq~&\int_{\left(B_{R_2}^{n+1}\setminus B_{R_1}^{n+1}\right)\cap\R}t^{1-2s}r^{2s-n}\left(\frac{\partial\ou^\infty}{\partial r}+\frac{2s}{r}\right)^2dxdt.
		\end{aligned}
		\end{equation*}
		Notice that in the last inequality we only used the weak convergence of $\ou^{\l_i}$
		to $\ou^\infty$ in $H^1_{\mathrm{loc}}(\overline\R,t^{1-2s}dxdt)$. So,
		$$\frac{\partial\ou^\infty}{\partial r}+\frac{2s}{r}=0\quad \mbox{a.e. in}\quad \R.$$
		Thus we proved the claim. In addition, $\ou^\infty$ is also stable because the stability condition for $\ou^{\l_i}$ passes to the limit. Then by Theorem \ref{th3.1} we get that \eqref{1.stable} holds, a contradiction to \eqref{1.stable1}. This proves  Theorem \ref{th1.1}.
	\end{proof}	
	\noindent {\bf Remark 4.2.} The above arguments also work if one uses the  truncated energy functional $E_R$ as mentioned in Remark \ref{re1.1}.

\end{document}